\documentclass[10pt,twoside]{amsart}
\usepackage{amsmath,amssymb,amscd,amsfonts} 
\usepackage[all]{xy}
\numberwithin{equation}{section}
\theoremstyle{plain}
\newtheorem{thm}{Theorem}[section]
\newtheorem{prop}[thm]{Proposition}
\newtheorem{lem}[thm]{Lemma}

\newtheorem{rmk}[thm]{Remark}

\newcommand{\F}{\mathbb{F}}
\newcommand{\Fbar}{\overline{\F}}

\newcommand{\Q}{\mathbb{Q}}

\newcommand{\Z}{\mathbb{Z}}

\newcommand{\cM}{\mathcal{M}}

\newcommand{\cO}{\mathcal{O}}

\newcommand{\Hom}{\mathrm{Hom}}

\newcommand{\pr}{\mathrm{pr}}

\newcommand{\Spec}{\mathrm{Spec}\,}

\newcommand{\inj}{\hookrightarrow}

\newcommand{\xto}[1]{\xrightarrow{#1}}

\newcommand{\etale}{{\mathrm{et}}}

\newcommand{\cf}{cf.\,}
\newcommand{\red}{\mathrm{red}}

\newcommand{\CH}{\mathrm{CH}}

\newcommand{\Map}{\mathrm{Map}}
\newcommand{\Coker}{\mathrm{Coker}}
\newcommand{\BM}{\mathrm{BM}}

\title[On two higher Chow groups]{On two higher Chow groups of schemes over a finite field}
\author{Satoshi Kondo}
\address{Institute for the Physics and Mathematics of the 
Universe, The University of Tokyo, 5-1-5 Kashiwa-no-ha, 
Kashiwa-shi Chiba 277-8568, JAPAN}
\email{satoshi.kondo@ipmu.jp}
\author{Seidai Yasuda}
\address{Department of Mathematics, Graduate School of Science,
Osaka University, Toyonaka, Osaka 560-8502, JAPAN}
\email{s-yasuda@math.sci.osaka-u.ac.jp}
\begin{document}
\begin{abstract}
Given a separated scheme $X$ of finite type over a finite field,
its higher Chow groups $CH_{-1}(X, 1)$ and $CH_{-2}(X, 3)$ are
computed explicitly.
\end{abstract}
\maketitle

\section{Introduction}\label{sec: Introduction}
Let $\F_q$ be the field of $q$ elements of characteristic $p$.
For a separated scheme $X$ which is essentially 
of finite type over $\Spec\F_q$, 
we define the Borel-Moore motivic
homology group $H^\BM_i(X,\Z(j))$ as
the homology group $H_{i-2j}(z_j(X, \bullet))=\CH_j(X, i-2j)$ 
of Bloch's cycle complex $z_j(X,\bullet)$ 
(\cite[Introduction, p.~267]{Bloch} see also 
\cite[2.5, p.~60]{Ge-Le2} to remove the
condition that $X$ is quasi-projective; see \cite{Levine} for 
the labeling using dimension and not codimension). 
If $j > i$ or $j >\dim\, X$, then
$H^\BM_i(X,\Z(j)) =0$ for trivial reason.
When $X$ is essentially smooth
over $\Spec\F_q$, it coincides with the motivic cohomology
group defined in \cite[Part I, Chapter I, 2.2.7, p.~21]{Levine} or
\cite{Voevodsky} (\cf \cite[Theorem 1.2, p.~300]{Levine2}, 
\cite[Corollary 2, p.~351]{Voevodsky2}).
For an abelian group $M$, we set
$H^\BM_i(X, M(j))=H_{i-2j}(z_j(X,\bullet)\otimes_\Z M)$.

For a scheme $X$, we let $\cO(X)=H^0(X,\cO_X)$.
The aim of this paper is to prove the following theorem.
\begin{thm}\label{prop:AppB_main}
Let $X$ be a connected scheme 
which is separated and of finite type over $\Spec\F_q$.
Then for $j=-1,-2$, the pushforward map 
$$
\alpha_X : 
H^\BM_{-1}(X,\Z(j)) \to H^\BM_{-1}(\Spec\cO(X),\Z(j))
$$ is an isomorphism if $X$ is proper, and 
the group $H^\BM_{-1}(X,\Z(j))$ is zero if $X$ is
not proper.
\end{thm}

Theorem~\ref{prop:AppB_main} is a generalization of 
a theorem of Akhtar 
\cite[Theorem 3.1, p.~285]{Akhtar} where the claim
is proved for $j=-1$ and $X$ smooth projective over $\Spec \F_q$.
Our proof of Theorem~\ref{prop:AppB_main}
is independent of \cite{Akhtar}, and we
do not require a Bertini-type theorem.

If we assume Parshin's  conjecture,
then the statement in the theorem holds for any $j \le -1$. 
Moreover we also obtain 
$H_i^\BM(X, \Z(j))=0$ for
any $i \le -2$ and $j \le -1$.
The method is explained in Section~\ref{sec:Parshin}

We define the \'etale Borel-Moore (not motivic) 
homology with $\Z_\ell$-coefficients,
where $\ell$ is a prime different from $p$,
in Remark~\ref{rmk:BMetale}. 
Then we compute it explicitly,
and deduce that
$H_i^\BM(X, \Z(j))\otimes_\Z \Z_\ell \cong 
H_i^{\BM, \mathrm{et}}(X, \Z_\ell(j))$
in the range $i \le -1$ and $j \le -1$
(using Parshin's conjecture where it is needed for 
the computation of the Borel-Moore motivic homology groups).

The original version of this paper was written without using 
the Bloch-Kato-Milnor conjecture.   
We use it as a theorem of Rost and Voevodsky.
It is used via theorems 
of Geisser and Levine (e.g., \cite[Corollary 1.2, p.56]{Ge-Le2}).

{\bf Acknowledgment}
The first author thanks Shuji Saito and Thomas Geisser 
for valuable comments in the course of revision.
The second author would like to thank Akio Tamagawa for 
helpful suggestions on the proof of a lemma in the earlier version
which no longer exists.
The authors thank Thomas Geisser and Tohru Korita
for pointing out a mistake in a former version of the manuscript.

We thank the referee for numerous 
suggestions which shortened this paper and clarified many things
at many points.

During this research, the first author was supported as a Twenty-First Century COE 
Kyoto Mathematics Fellow and was partially supported by JSPS Grant-in-Aid for 
Scientific Research 17740016. The second author was partially supported by JSPS 
Grant-in-Aid for Scientific Research 16244120, 21540013, and 24540018.

\section{Higher Chow groups of smooth curves over a finite field}
\label{sec:curves}
A curve will mean a scheme of pure dimension one, separated and of finite type over
a field.
The aim of this section is to compute the higher Chow groups
$\CH^i(X,j)$ for a smooth curve $X$ over a finite field in the range $i,j\ge 0$.
\begin{lem}\label{lem:q=0}
Let $X$ be a connected smooth curve 
over a finite field.  Then
\[
\CH^i(X,0)\cong
\left\{
\begin{array}{ll}
\Z,&i=0,\\
\mathrm{Pic}(X), &i=1, \\
0, &i\ge 2.
\end{array}
\right.
\]
\end{lem}
\begin{proof}
These are the classical Chow groups 
and the computation is known.  
For $i \ge 2$, it vanishes by dimension reason.
See also \cite[THEOREM (6.1), p.287]{Bloch}.
\end{proof}

\begin{lem}\label{lem:q ge 2}
Let $j \ge 2$.
Let $X$ be a smooth curve over a finite field
$\F_q$ of characteristic $p$.  Then
we have $\CH^i(X,j)=0$ for $i > j$,
and 
for $i \le j$, the cycle map in
\cite[(1.2), p.56]{Ge-Le2}
gives an isomorphism
\[
\CH^i(X,j)\cong
\bigoplus_{\ell\neq p}
H_{\mathrm{et}}^{2i-j-1}(X,\Q_\ell/\Z_\ell(i)).
\]
The right hand side is zero unless $2i-j=1,2,3$.
If moreover $X$ is affine, the right hand side is zero 
for $2i-j=3$.
\end{lem}
\begin{proof}
We first note that $\CH^i(X,j)=0$
if $i>j+1$ by dimension reason.
Henceforth we consider the case $i\le j+1$.

Recall Bloch's formula (\cite[THEOREM(9.1), p.296]{Bloch}):
\begin{equation}
\label{eqn:Bloch}
K_j(X)\otimes_\Z \Q \cong 
\bigoplus_{i\ge 0} \CH^i(X,j)\otimes_\Z \Q
\end{equation}
where $K_j$ is the $j$-th algebraic K-group.
Recall also Harder's result (the result 
\cite[3.2.5 Korollar, p.175] {Ha} is
not correctly stated; we refer to \cite[THEOREM 0.5, p.70]{Grayson}
and the remark there for the explanation and the 
corrected statement) 
which implies that
$K_j(X) \otimes_\Z \Q=0$ for $j\ge 2$.
Hence $\CH^i(X,j)$ is torsion for $j \ge 2$.

We recall the definition of motivic cohomology given in 
\cite[Section 2.5, p.60]{Ge-Le2}.
For a smooth scheme $X$ over a field, 
define the cohomological cycle complex by 
$Z^j(X,i)=z^j(X,2j-i)$ where $z^*(-, *)$ is Bloch's cycle complex
(\cite[INTRODUCTION, p.267]{Bloch}, see also \cite[2.2, p.58]{Ge-Le2}).
Then, for an abelian group $A$,  
define $H_\cM^j(X, A(i))=H^i(Z^j(X, *)\otimes_\Z A)$.
We have $H_\cM^{2i-j}(X, \Z(i))=\CH^i(X,j)$.

The exact sequence 
$0\to \Z \to \Q \to \Q/\Z \to 0$ 
gives an exact sequence
\[
\begin{array}{l}
H_\cM^{2i-j-1}(X, \Q(i)) 
\to H_\cM^{2i-j-1}(X, \Q/\Z(i))
\xto{(1)}
H_\cM^{2i-j}(X, \Z(i))
\\
\to
H_\cM^{2i-j}(X, \Q(i)).
\end{array}
\]
The first and the last terms are zero as was remarked above.
The map (1) is hence an isomorphism.

Since we are in the range $i \le j+1$ (equivalently, $2i-j-1 \le i$),
we apply \cite[Corollary 1.2, p.56]{Ge-Le2}
and the computation (\cite[Theorem 8.4, p.491]{Ge-Le1}) 
by Geisser and Levine of $p$-torsion motivic cohomology
to obtain
\[
H_\cM^{2i-j-1}(X, \Q/\Z(i))
\cong
\bigoplus_{\ell\neq p}H_{\mathrm{et}}^{2i-j-1}(X,\Q_\ell/\Z_\ell(i))
\oplus \varinjlim_r H^{i-j-1} (X_\mathrm{Zar}, \nu_r^i).
\]
(We refer to \cite{Ge-Le1} for the definition of 
$H^*(X_\mathrm{Zar}, \nu_r^i)$.)
One can compute the right hand side explicitly.
The $p$-part is zero since we are in the range $i \le j+1$ and $j \ge 2$.

Set $a=2i-j-1$.
Let us show that 
$H_{\mathrm{et}}^a(X, \Q_\ell/\Z_\ell(i))=0$ if $a \ge 3$.
If $a \ge 4$, it follows from the fact that 
the cohomological dimension of a curve over a finite field is 3
(\cite[Expos\'e X, Corollaire~4.3, p.15]{SGA4}).
Suppose $a=3$.
We have an exact sequence
\[
H_{\mathrm{et}}^3(X, \Q_\ell(i))
\to 
H_{\mathrm{et}}^3(X, \Q_\ell/\Z_\ell(i))
\to
H_{\mathrm{et}}^4(X, \Z_\ell(i)).
\]
The third term is zero because of the cohomological dimension reason.
The Hochschild-Serre spectral sequence reads
\[
E_2^{p,q}=H^p_{\mathrm{et}}(\mathrm{Gal}(\overline{\F}_q/\F_q),
H_{\mathrm{et}}^q(\overline{X}, \Q_\ell(i))
\Rightarrow
H_{\mathrm{et}}^{p+q}(X, \Q_\ell(i))
\]
where $\overline{X}=X\times_{\Spec \F_q} \Spec \overline{\F}_q$.
We have $E_2^{0,3}=0$ since 
$H_\mathrm{et}^3(\overline{X}, \Q_\ell(i))=0$.
To show $E_2^{1,2}=0$, note that
the weight of 
$H^2_\mathrm{et}(\overline{X}, \Q_\ell(i))$ 
is $2-2i$.  Since $j \ge 2$ and $a=3$, the weight $2-2i$ 
is nonzero,
hence $E_2^{1,2}=0$.
We have $E_2^{2,1}=0$ because the 
cohomological dimension of 
$\mathrm{Gal}(\overline{\F}_q/\F_q)$ is one.
This proves the claim in this case.

Suppose $a=2$ and $X$ is affine.
We have an exact sequence
\[
H_{\mathrm{et}}^2(X, \Q_\ell(i))
\to 
H_{\mathrm{et}}^2(X, \Q_\ell/\Z_\ell(i))
\to
H_{\mathrm{et}}^3(X, \Z_\ell(i)).
\]
The third term is zero since 
the cohomological dimension 
of an affine curve over a finite field is 2
(\cite[Expos\'e XIV, Th\'eor\`eme~3.1, p.15]{SGA4}).
We use the Hochshild-Serre spectral sequence
as above.
We have $E_2^{0,2}=0$ using the same 
cohomological dimension reason.
Note that the (possible) weights of 
$H^1_{\mathrm{et}}(\overline{X}, \Q_\ell(i))$
are $1-2i$ and $2-2i$.
Since neither of them is zero,
we have $E_2^{1,1}=0$.
These imply that $H^2_{\mathrm{et}}(X, \Q_\ell(i))=0$
hence the claim in this case.
\end{proof}

\begin{lem}\label{lem:q=1}
Let $X$ be a smooth curve over a finite field.
We have
\[
\CH^i(X,1)=\left\{
\begin{array}{ll}
0 ,&i=0,\\
\cO(X)^\times, &i=1, \\
0, &i\ge 3.
\end{array}
\right.
\]
\end{lem}
\begin{proof}
The case $i=0$ is trivial. 
The case $i=1$ is found in Bloch's paper (\cite[THEOREM (6.1), p.287]{Bloch}).
For $i \ge 3$, the claim follows by dimension reason.
\end{proof}
\begin{lem}\label{lem:3,2}
Let $U$ be an affine smooth curve over a finite field. 
Then $\CH^2(U, 1)=0$.
\end{lem}

\begin{proof}
The group $SK_1(U)$ sits in the following exact sequence:
\[
0 \to SK_1(U) \to K_1(U) \to \cO(U)^\times \to 0.
\]
We use the result 
\cite[THEOREM~0.5, p.70]{Grayson},
which says that $SK_1(U)\otimes_\Z \Q=0$
for an affine smooth curve $U$.
Using Lemma~\ref{lem:q=1}, it follows from counting the dimension of 
both sides of Bloch's formula \eqref{eqn:Bloch}
that $\mathrm{dim}_\Q\, H_\cM^3(U, \Q(2))=0$.
For the rest of the proof, one proceeds as in Lemma~\ref{lem:q ge 2}.
\end{proof}

\begin{lem}\label{rmk:cycle_class}
Let $Z$ be a scheme which is finite over $\Spec\F_q$.
Then we have isomorphisms
$$
\begin{array}{l}
H^\BM_{-1}(Z,\Z(j)) \stackrel{(1)}{\cong}
H^\BM_{-1}(Z_\red,\Z(j)) \\
\stackrel{(2)}{\cong}
H^\BM_0(Z_\red,\Q/\Z(j)) 
\stackrel{(3)}{\cong}
\bigoplus_{\ell \neq p} 
H^0_\etale(Z_\red, \Q_\ell/\Z_\ell(-j))
\end{array}
$$
for $j\le -1$, which are functorial with respect to pushforwards.
Here $Z_{\mathrm{red}}$
denotes the reduced scheme associated to $Z$.
\end{lem}
\begin{proof}
For any scheme $W$ of finite type over $\F_q$
and an abelian group $A$,
we have $H_i^\BM(W, A(j)) \cong
H_i^\BM(W_{\mathrm{red}}, A(j))$ for any $i,j$,
since the cycle complexes are canonically
isomorphic by definition. This gives the isomorphism (1).

For (2),  
we use the long exact sequence
of the universal coefficient theorem for higher Chow groups:
\[
\begin{array}{l}
H^\BM_0(Z_\red,\Q(j)) \to H^\BM_0(Z_\red,\Q/\Z(j)) \\
\to
H^\BM_{-1}(Z_\red,\Z (j)) \to H^\BM_{-1}(Z_\red, \Q(j)).
\end{array}
\]
We know that the higher K-groups of a finite field are torsion from
\cite[THEOREM~8, p.583]{QuFinite}.  
Then using a formula of Bloch \eqref{eqn:Bloch}, 
we see that the groups in the sequence above with $\Q$-coefficient
are zero.

Since $\F_q$ is perfect, $\Spec Z_\red$
is smooth over $\F_q$. The map (3) is the cycle map in \cite{Ge-Le2}
(which is defined for smooth schemes over a field).
The fact that the cycle map is an isomorphism follows from
\cite[Corollary 1.2, p.~56]{Ge-Le2} and \cite[(11.5), THEOREM, p.~328]{Me-Su1},
and \cite[Theorem~1.1, p406]{Ge-Le1}.

It is clear that the isomorphisms (1) and (2) are functorial 
with respect to pushforwards.
Let $Z'$ be another scheme which is finite over $\Spec\F_q$
and let $f:Z' \to Z$ be a morphism over $\Spec\F_q$.
We prove that the isomorphisms (3) for $Z$ and $Z'$ are 
compatible with the pushforward maps with respect to $f$.
We are easily reduced to the case when both $Z$ and
$Z'$ are spectra of finite extensions of $\F_q$.
Let $Z'' = Z' \times_{Z} Z'$ and let
$\pr_1,\pr_2 : Z'' \to Z'$ denote the projections to the
first and the second factor, respectively.
Then the diagram
$$
\begin{CD}
H_0^\BM(Z',\Q/\Z(j)) @>{f_*}>> H_0^\BM(Z,\Q/\Z(j))\\
@V{\pr_2^*}VV @VV{f^*}V \\
H_0^\BM(Z'',\Q/\Z(j)) @>{(\pr_1)_*}>> H_0^\BM(Z',\Q/\Z(j))
\end{CD}
$$
is commutative, and a similar commutativity holds for
the corresponding \'etale cohomology groups. Since the pullback map
$f^* :H^0_\etale(Z,\Q_\ell/\Z_\ell(-j))
\to H^0_\etale(Z',\Q_\ell/\Z_\ell(-j))$ is injective, it suffices to prove that
the isomorphisms (3) for $Z''$ and $Z'$ are 
compatible with the pushforward maps with respect to $\pr_1$.
Since $Z''$ is isomorphic to the disjoin union of a finite number
of copies of $Z'$, the last claim can be checked easily.
The lemma is proved.
\end{proof}

The statement is better understood using \'etale Borel-Moore homology groups. 
See Remark~\ref{rmk:BMetale}.

\begin{rmk}\label{rmk:cyclic}
Suppose $X$ is a connected scheme 
which is proper over $\Spec \F_q$.
Then Theorem~\ref{prop:AppB_main}
says that the group $H^\BM_{-1}(X,\Z(j))$ 
is isomorphic to $H^\BM_{-1}(\Spec \cO(X), \Z(j)).$
We can then use Lemma~\ref{rmk:cycle_class} and 
compute this group explicitly.  The computation 
of \'etale cohomology group  shows that this group 
is   
a cyclic group whose order is
$$|\cO(X_\red)|^{-j} -1$$ for each $j=-1, -2$.
\end{rmk}

\begin{lem}\label{lem:norm_surj}
Let $\F'_{1}$, $\F'_{2}$ be two finite extensions of $\F_q$
with $\F'_{1} \subset \F'_{2}$.
Then for $j \le -1$, 
the pushforward map $H^\BM_{-1}(\Spec \F'_{2}, \Z(j)) \to
H^\BM_{-1}(\Spec \F'_{1}, \Z(j))$ is surjective.
\end{lem}
\begin{proof}
By Lemma~\ref{rmk:cycle_class}, the cycle class map gives 
an isomorphism $\alpha : H^\BM_{-1}(\Spec \F'_k,\Z(j)) 
\cong \bigoplus_{\ell \neq p} 
H_\etale^0(\Spec \, \F'_k, \Q_\ell/\Z_\ell(-j))$
for $k=1,2$. The cycle class map is compatible with
the pushforward by a finite morphism (\cite[Lemma~3.5(2), p.69]{Ge-Le2}).
Thus the claim follows from the corresponding statement
for the \'etale cohomology groups.
(See \cite[Lemme~6 iii), p.269]{Soule} and 
\cite[IV.1.7, p.283]{Soule}.)
\end{proof}

\section{Proof of Theorem 1.1}
\label{sec:proof main}

\begin{lem}\label{lem:normal_gcd}
Let $X$ be an integral scheme which is of
finite type over $\Spec \F_q$. Let $\F$ be the
algebraic closure of $\F_q$ in $\cO(X)$.
Then the degree $[\F:\F_q]$ divides the degree
$[\kappa(x): \F_q]$ for
all closed points $x\in X_0$. If moreover $X$ is normal,
we have the equality 
$[\F:\F_q] = \gcd_{x\in X_0} [\kappa(x): \F_q]$.
\end{lem}

\begin{proof}
For each $x \in X_0$, the composite
$\F \inj \cO(X) \to \kappa(x)$,
where the second map is induced from the pullback map
by the closed immersion,
is injective since $\F$ is a field. Hence
$[\F:\F_q]$ divides $[\kappa(x): \F_q]$.

Suppose that $X$ is normal of dimension $d$. 
Then $X$ is geometrically irreducible
as a scheme over $\Spec \F$.
Let $d_0=(\gcd_{x \in X_0} [\kappa(x):\F_q])/[\F:\F_q]$.
Let $\F \subset \F_0$ denote an extension of degree $d_0$.
The canonical morphism
$f:X \times_{\Spec \F} \Spec \F_0 \to X$
is a finite \'etale cover in which every closed point of $X$
splits completely.
It follows from a Chebotarev density type theorem 
(\cite{La}; we refer to \cite[Lemma 1.7, p.98]{Raskind}
for the statement which is ready for our use)
that $f$ is an isomorphism.  Hence $d_0=1$.
This completes the proof.
\end{proof}

\begin{lem}\label{lem:red_to_non_proper}
Let $d \ge 0$ be an integer.
Suppose that Theorem~\ref{prop:AppB_main} holds
for all connected normal 
schemes over $\Spec \F_q$ of dimension $d$
which are not proper over $\Spec \F_q$.
Then Theorem~\ref{prop:AppB_main} holds
for all connected normal schemes 
of dimension $d$ which are proper over $\Spec \F_q$.
\end{lem}
\begin{proof}
Let $X$ be a connected normal scheme 
of dimension $d$ which is
proper over $\Spec \F_q$. 
Let $j \in \{-1, -2\}$. 
Let $x \in X_0$ be a closed point. 
The pushforward map $\alpha_X$ in the statement of
Theorem~\ref{prop:AppB_main}
is surjective since its composite with
the pushforward map 
$\iota_{x*}: H^\BM_{-1}(\Spec \kappa(x),\Z(j))
\to H^\BM_{-1}(X,\Z(j))$ is surjective
by Lemma~\ref{lem:norm_surj}.
By assumption, 
the group $H^\BM_{-1}(X \setminus \{x\},\Z(j))$
is zero. Hence the localization sequence shows that 
the pushforward map 
$\iota_{x*}$ is surjective. 
This implies that
$H^\BM_{-1}(X,\Z(j))$ is of order dividing
$\gcd_{x \in X_0}(|\kappa(x)|^{-j} -1) 
= q^{(-j)\cdot \gcd_{x\in X_0}[\kappa(x):\F_q]}- 1$.
This equals
$q^{[\F:\F_q]}$ by Lemma~\ref{lem:normal_gcd}
where $\F$ is the algebraic closure of $\F_q$ in $\cO(X)$.
We know the order of the target of $\alpha_X$ is also equal 
to this value from Lemma~\ref{rmk:cycle_class}.  
Hence the bijectivity of $\alpha_X$ follows.
\end{proof}

\begin{lem}\label{lem:sets}
Let
$S_1 \xleftarrow{\alpha_1} S_3 \xto{\alpha_2} S_2$ be a diagram of sets and
let $R$ be an integral domain.
For $i=1,2,3$, let $\Map(S_i,R)$ denote the $R$-module of $R$-valued functions on $S_i$.
Then the cokernel of the homomorphism
$$
\beta : \Map(S_1,R) \oplus \Map(S_2,R)\to \Map(S_3,R)
$$
which sends $(f_1,f_2)$ to 
$f_1 \circ \alpha_1 - f_2 \circ \alpha_2$ is $R$-torsion free.
\end{lem}

\begin{proof}
Let $e:\Map(S_3,R) \to \Coker\, \beta$ denote the quotient map.
Let $f \in \Map(S_3,R)$ and suppose that $e(f)$ is
an $R$-torsion element in $\Coker\, \beta$. We prove that $e(f)=0$.
Since $e(f)$ is an $R$-torsion element, there exist a non-zero element $a \in R$
and an element $(f_1,f_2) \in \Map(S_1,\Z) \oplus \Map(S_2,\Z)$
satisfying $a f = f_1 \circ \alpha_1 - f_2 \circ \alpha_2$.
Let us take a complete set $T \subset R$ of representatives of $R/aR$.
For $i=1,2$, let $\overline{f}_i$ denote the unique $T$-valued function
on $S_i$ satisfying $\overline{f}_i(x)
\equiv f_i(x) \mod{aR}$ for every $x \in S_i$.
We then have
$$
\overline{f}_1 \circ \alpha_1 \equiv
f_1 \circ \alpha_1 \equiv f_2 \circ \alpha_2 \equiv
\overline{f}_2 \circ \alpha_2
$$
modulo $a \Map(S_3,R)$.
Since both $\overline{f}_1 \circ \alpha_1$
and $\overline{f}_2 \circ \alpha_2$ are $T$-valued functions,
we have $\overline{f}_1 \circ \alpha_1 = \overline{f}_2 \circ \alpha_2$.
For $i=1,2$, let $g_i$ denote the unique $R$-valued function
on $S_i$ satisfying $f_i = \overline{f}_i + a g_i$.
Then
$$
a f = (\overline{f}_1+ a g_1) \circ \alpha_1 - (\overline{f}_2+ a g_2) \circ \alpha_2
= a (g_1 \circ \alpha_1 - g_2 \circ \alpha_2).
$$
Since $\Map(S_3,R)$ is $R$-torsion free, we have
$f = g_1 \circ \alpha_1 - g_2 \circ \alpha_2$.
This shows that $e(f)=0$, which proves the claim.
\end{proof}

\begin{lem}\label{lem:red_to_normal}
Let $d \ge 0$ be an integer.
Suppose that Theorem~\ref{prop:AppB_main} holds
for all connected schemes of dimension smaller than $d$
which are proper over $\Spec\F_q$
and for all connected normal schemes 
over $\Spec \,\F_q$ of dimension $d$.
Then Theorem~\ref{prop:AppB_main} holds
for all connected schemes of dimension $d$
which are proper over $\Spec \F_q$.
\end{lem}
\begin{proof}
Let $X$ be a connected scheme of dimension $d$ 
which is proper over $\Spec \F_q$. 
Without loss of generality we may assume that
$X$ is reduced. Suppose that $X$ is not normal.
Let $\pi:X' \to X$ denote the normalization of $X$.
The scheme $X'$ is proper over $\Spec \F_q$
since $\pi$ is finite by \cite[Remarque 6.3.10, p.~120]{EGA2}. 
Take a reduced closed subscheme $Y \subset X$
of dimension less than that of $X$ such that $X \setminus Y$
is normal and set $Y' = (Y\times_X X')_\red$. 
By assumption, Theorem~\ref{prop:AppB_main} holds
for each connected component 
of $X \setminus Y$, $X'$, $Y$ and $Y'$.

Let $j \in \{-1, -2\}$. 
Let us consider the commutative diagram
$$
\begin{CD}
H^\BM_{-1}(Y,\Z(j)) @>{\beta}>>
H^\BM_{-1}(X,\Z(j)) \\
@V{\alpha_Y}V{\cong}V @V{\alpha_X}VV \\
H^\BM_{-1}(\Spec \cO(Y), \Z(j))
@>{\gamma}>>
H^\BM_{-1}(\Spec \cO(X),\Z(j))
\end{CD}
$$
where all the homomorphisms are pushforwards.
Since $\alpha_Y$ is an isomorphism and
we know that $\gamma$ is surjective using
Lemma~\ref{rmk:cycle_class} and Lemma~\ref{lem:norm_surj}, 
the homomorphism $\alpha_X$ is surjective.

Since $H^\BM_{-1}(X \setminus Y,\Z(j))$ is zero,
the localization sequence shows that
the map $\beta$ is surjective.

We use the following notation for short: for a scheme $Z$, we denote
$\Spec \cO(Z)_\red$ by $a(Z)$.
\begin{lem}
\label{lem:et cartesian}
The diagram
$$
\begin{CD}
H^0_\etale(a(Y'),\Q_\ell/\Z_\ell(-j)) 
@>>> H^0_\etale(a(X'),\Q_\ell/\Z_\ell(-j)) \\
@VVV @VVV \\
H^0_\etale(a(Y),\Q_\ell/\Z_\ell(-j)) 
@>>> H^0_\etale(a(X),\Q_\ell/\Z_\ell(-j)), \\
\end{CD}
$$
where the arrows are the pushforward homomorphisms,
is cocartesian for every prime number $\ell \neq p$.
\end{lem}
\begin{proof}
Let $\overline{X}=X \times_{\Spec \F_q}\Spec \overline{\F}_q$
and define $\overline{Y}$, $\overline{X}'$ and $\overline{Y}'$
in a similar manner.
By \cite[(1.4.16.1), p.94]{EGAIII}, 
we have $a(\overline{X}) = a(X) \times_{\Spec \F_q}\Spec \overline{\F}_q$
and similarly for $Y$, $X'$, and $Y'$.
Since $j \neq 0$, the weight argument shows that 
the $\mathrm{Gal}(\overline{\F}_q/\F_q)$-coinvariants
of any quotient $\mathrm{Gal}(\overline{\F}_q/\F_q)$-module 
and of any divisible $\mathrm{Gal}(\overline{\F}_q/\F_q)$-submodule of
$H^0_\etale(a(\overline{Y}'),\Q_\ell/\Z_\ell(-j))$ vanish.
Hence it suffices to show that the diagram
$$
\begin{CD}
H^0_\etale(a(\overline{Y}'),\Q_\ell/\Z_\ell(-j)) 
@>>> H^0_\etale(a(\overline{X}'),\Q_\ell/\Z_\ell(-j)) \\
@VVV @VVV \\
H^0_\etale(a(\overline{Y}),\Q_\ell/\Z_\ell(-j)) 
@>>> H^0_\etale(a(\overline{X}),\Q_\ell/\Z_\ell(-j)) 
\end{CD}
$$
is cocartesian in the category 
of $\mathrm{Gal}(\overline{\F}_q/\F_q)$-modules and that
the kernel of the homomorphism 
$H^0_\etale(a(\overline{Y}'),\Q_\ell/\Z_\ell(-j))
\to H^0_\etale(a(\overline{X}'),\Q_\ell/\Z_\ell(-j))
\oplus H^0_\etale(a(\overline{Y}),\Q_\ell/\Z_\ell(-j))$
is divisible.
By taking the Pontryagin dual, we prove that the diagram
$$
\begin{CD}
H^0_\etale(a(\overline{X}),\Z_\ell(j)) 
@>>> H^0_\etale(a(\overline{X}'),\Z_\ell(j)) \\
@VVV @VVV \\
H^0_\etale(a(\overline{Y}),\Z_\ell(j)) 
@>>> H^0_\etale(a(\overline{Y}'),\Z_\ell(j)),
\end{CD}
$$
where the arrows are the pullback homomorphisms,
is cartesian in the category 
of $\mathrm{Gal}(\overline{\F}_q/\F_q)$-modules
and that the cokernel of the homomorphism
\begin{equation}\label{eq:beta}
H^0_\etale(a(\overline{X}'),\Z_\ell(j)) 
\oplus H^0_\etale(a(\overline{Y}),\Z_\ell(j)) 
\to H^0_\etale(a(\overline{Y}'),\Z_\ell(j))
\end{equation}
is torsion free.

Let $Z$ be a scheme which is of finite type over $\F_q$.
Let us write $\overline{Z}= Z \times_{\Spec \F_q} \Spec \overline{\F}_q$.
Then we have an isomorphism
\begin{equation}
\label{eqn:map isom}
H^0_{\mathrm{et}}(a(\overline{Z}), \Z_\ell(j))
\cong \mathrm{Map}(\pi_0(\overline{Z}),
\Z_\ell) \otimes_{\Z_\ell} \Z_\ell(j)
\end{equation}
of $\mathrm{Gal}(\overline{\F}_q/\F_q)$-modules,
which is functorial with respect to pullbacks. 
It then follows from Lemma \ref{lem:sets} that
the homomorphism \eqref{eq:beta} has a torsion free cokernel.
Hence it suffices to show that 
the diagram
\begin{equation}\label{dia:3}
\begin{CD}
\pi_0(\overline{X}) @<<< \pi_0(\overline{X}') \\
@AAA @AA{\varphi}A \\
\pi_0(\overline{Y}) @<{\psi}<< \pi_0(\overline{Y}') \\
\end{CD}
\end{equation}
is cocartesian in the category of sets.

As $X' \to X$ is a normalization,
it is surjective.
As surjectivity is preserved under base change,
the map
$\overline{Y}' \to \overline{Y}$
is surjective, hence $\psi$ is surjective.
This implies that the pushout of 
the diagram
$$
\pi_0(\overline{X}')  \xleftarrow{\varphi} 
 \pi_0(\overline{Y}') \xto{\psi}
  \pi_0(\overline{Y}) 
$$
is isomorphic to the quotient of $\pi_0(\overline{X}')$
by the following equivalence relation.
We define a binary relation $\sim$
on $\pi_0(\overline{X}')$
as follows.  We say $x_1' \sim x_2'$
if there exist
$y_1', y_2' \in \pi_0(\overline{Y}')$
such that $x_1'=\varphi(y_1'),
x_2'=\varphi(y_2')$, and
$\psi(y_1')=\psi(y_2')$.
We also write $\sim$ for the 
equivalence relation on $\pi_0(\overline{X}')$
generated by the binary relation above.
Let us show that the map 
$\phi: \pi_0(\overline{X}')/\sim \to \pi_0(\overline{X})$
obtained from the diagram \eqref{dia:3}
is an isomorphism.

As the \'etale base change of a normalization,
$\overline{X}'\to \overline{X}$ is a normalization.
Hence $\pi_0(\overline{X}')$ coincides with the 
set of irreducible components of $\overline{X}$.
As a normalization is a surjective morphism,
the map $\phi$ 
is surjective.

Let $C_1, C_2$ be two distinct irreducible components of $\overline{X}$.
We claim that if $C_1\cap C_2 \neq \emptyset$
then the classes of $C_1$ and $C_2$ in $\pi_0(\overline{X}')/\sim$
coincide.
Let $y \in C_1 \cap C_2$.  
Then the local ring $\cO_{\overline{X},y}$ is not
an integral domain. 
Since we chose $Y$ so that $X\setminus Y$ is normal,
$y$ belongs to $\overline{Y}$.
One can take $y_1, y_2 \in \overline{X}'$ 
lying over $y$ such that $y_i$ 
lies in the same connected component as $C_i$ for each $i=1,2$.
Note that  $y_1, y_2 \in \overline{Y}'$ since
they both lie over $y \in \overline{Y}$.
Then using the definition of the equivalence relation
above for $y_1$ and $y_2$, we see that $C_1\sim C_2$.

Let $C_1'$ and $C_2'$ be two irreducible components
of $\overline{X}$.
It follows from the discussion above that
if they belong to the same connected component,
then $C_1' \sim C_2'$.
This implies the injectivity of $\phi$.
This proves the claim.
\end{proof}

We return to the proof of Lemma \ref{lem:red_to_normal}.
It follows from Lemma~\ref{rmk:cycle_class}
and Lemma~\ref{lem:et cartesian} that
the diagram
$$
\begin{CD}
H^\BM_{-1}(\Spec \cO(Y'),\Z(j)) 
@>>>  
H^\BM_{-1}(\Spec \cO(X'),\Z(j)) \\
@VVV @VVV \\
H^\BM_{-1}(\Spec \cO(Y),\Z(j)) 
@>>> 
H^\BM_{-1}(\Spec \cO(X),\Z(j)) \\
\end{CD}
$$
is cocartesian.   
We saw that $\gamma$ is surjective. 
Taking a lift and composing with $\beta \circ (\alpha_Y)^{-1}$
we obtain a map $H^\BM_{-1}(\Spec \cO(X),\Z(j)) \to H^\BM_{-1}(X, \Z(j))$.
The fact that the diagram above is cocartesian and some diagram chasing
imply that this map does not depend on the choice of a lift and this map 
is a homomorphism.  We then see that the homomorphism
$\beta$ factors through the homomorphism
$$
\begin{array}{l}
H^\BM_{-1}(Y,\Z(j)) 
\xto{\alpha_Y} 
H^\BM_{-1}(\Spec \cO(Y),\Z(j)) \\
\xto{\gamma}
H^\BM_{-1}(\Spec \cO(X),\Z(j)).
\end{array}
$$
This proves that the order of $H^\BM_{-1}(X,\Z(j))$ 
divides the order of
$$
H^\BM_{-1}(\Spec \cO(X),\Z(j)).
$$ 
Hence 
$\alpha_X$ is an isomorphism.
This completes the proof.
\end{proof}

\begin{lem}
\label{lem:not proper}
Let $U$ be a nonempty open subscheme of a separated connected scheme
$V$ over $\Spec \F_q$
such that $V \setminus U \neq \emptyset$.
Then $U$ is not proper over $\Spec \F_q$.
\end{lem}
\begin{proof}
As $V$ is separated,
the diagonal $\Delta \subset V \times_{\Spec \F_q} V$
is closed,
hence the restriction 
$\Delta \cap (U \times_{\Spec \F_q} V) \subset
U \times_{\Spec \F_q} V$ is closed.
The image of this closed set under the second projection
$U \times_{\Spec \F_q} V \to V$ is $U$, hence it 
is not closed in $V$ since $V$ is connected.
This shows 
the structure map
$U \to \Spec \F_q$
is not universally closed, hence it is not proper.
\end{proof}

\begin{proof}[Proof of Theorem~\ref{prop:AppB_main}]
First suppose $d=1$. The claim for $X$ normal and non-proper
follows from Lemmas~\ref{lem:3,2} and~\ref{lem:q ge 2}.
Then the claim for $X$ proper follows from
Lemmas~\ref{lem:red_to_non_proper} and \ref{lem:red_to_normal}.

Let  us prove the claim for non-proper $X$.
We use induction on the number of irreducible 
components $n$ of $X$.
Suppose $n=1$.
We may without loss of generality assume
$X$ is reduced so that $X$ is integral.
Take an open immersion from $X$ to a 
connected scheme $X'$ of dimension one
which is proper over $\Spec \F_q$
such that the complement $X' \setminus X$ is zero dimensional.
Let $j \in \{-1,-2\}$. We have proved that the pushforward map 
$H^\BM_{-1}(X',\Z(j)) \to 
H^\BM_{-1}(\Spec \cO(X'),\Z(j))$
is an isomorphism. 
This implies, using Lemma~\ref{lem:norm_surj}, that the pushforward map
$H^\BM_{-1}(X' \setminus X, \Z(j)) \to 
H^\BM_{-1}(X',\Z(j))$ is surjective. Hence,
by the localization sequence, we have 
$H^\BM_{-1}(X, \Z(j)) =0$.

Suppose $n \ge 2$.
We take a (non-empty) zero dimensional closed subscheme $Y \subset X$
such that 
$X \setminus Y=X_1 \coprod \dots \coprod X_r$ (disjoint union of schemes) 
with
the following properties:
\begin{enumerate}
\item
$X_i$ is a connected one dimensional open subscheme of $X$,
\item
the number of irreducible components of $X_i$ is less than $n$,
\item the closure $\overline{X}_i$ of $X_i$ in $X$ equals $X_i \cup Y$,
\end{enumerate}
for each $1 \le i \le r$.

We can for example take as $Y$ the following scheme.
Let $Y_0$ be a zero dimensional subscheme of $X$
such that the complement $X \setminus Y_0$
is not connected.  We order the set of such $Y_0$'s by inclusion,
and let $Y$ be a minimal one with respect to this 
ordering.  Let us check the properties (1)(2)(3).
Let $\{X_i\}_{1 \le i \le r}$
be the set  of  connected components of $X \setminus Y$ 
then (1) holds true.
We have $r \ge 2$ by construction. 
Since the number of irreducible components of $X$
equals the sum of the number of irreducible components
of the $X_i$'s, the property (2) holds true.
The closure $\overline{X}_i$ of $X_i$ in $X$ is contained 
in $X_i \cup Y$ by construction.  Suppose 
$\overline{X}_i \neq X_i \cup Y$
for some $1\le i \le r$.
Let $y \in (X_i \cup Y) \setminus \overline{X}_i$.
Then the minimality condition on the construction of $Y$
implies that 
$X \setminus (Y \setminus \{y\})=(X_1 \coprod \cdots \coprod X_r) \cup \{y\}$
is connected.
This implies in particular that $y \in \overline{X}_i$,
which is a contradiction, so (3) holds true.

Taking $U=X_i$ and $V=\overline{X}_i$ in Lemma~\ref{lem:not proper},
we see that $X_i$ is not proper.
By the non-properness of $X$
(and changing the indexing)
we may suppose that $\overline{X}_1$ is not proper.
The localization sequence gives the exact sequence
\[
H_0^\BM(X_1,\Z(j))
\xto{\varphi}
H_{-1}^\BM(Y,\Z(j))
\to
H_{-1}^\BM(\overline{X}_1, \Z(j)).
\]
By the inductive hypothesis,
we have 
$H_{-1}^\BM(\overline{X}_1, \Z(j))=0$,
hence $\varphi$ is a surjection.  Now use the 
following localization sequence
\[
\begin{array}{l}
\displaystyle\bigoplus_{i=1}^r H_0^\BM(X_i, \Z(j))
\xto{\psi}
H_{-1}^\BM(Y, \Z(j))
\to
H_{-1}^\BM(X, \Z(j))
\\
\to
\displaystyle\bigoplus_{i=1}^r 
H_{-1}^\BM(X_i, \Z(j)).
\end{array}
\]
Since $\varphi$ is surjective, $\psi$ is surjective.
By the inductive hypothesis,
$\bigoplus_{i=1}^r 
H_{-1}^\BM(X_i, \Z(j))=0$.
It follows that
$H_{-1}^\BM(X, \Z(j))=0$.
The claim is proved in the case $d=1$.

Next suppose that $d \ge 2$ and $X$ is affine.
Let $j \in \{-1, -2\}$.
The localization sequence gives an exact sequence
$$
\begin{array}{l}
\varinjlim_Y H^\BM_{-1}(Y,\Z(j))
\to H^\BM_{-1}(X,\Z(j))  \\
\to \varinjlim_Y H^\BM_{-1}(X\setminus Y,\Z(j)),
\end{array}
$$
where $Y$ runs over the reduced closed subschemes
of $X$ of pure codimension one. For dimension reasons, we have 
$\varinjlim_Y H^\BM_{-1}(X\setminus Y,\Z(j))=0$.
Hence by induction on $d$, we have $H^\BM_{-1}(X,\Z(j))=0$.

Next suppose that $d \ge 2$ and $X$ is not proper.
Using a similar argument as above, 
we are reduced, by induction on the
number of irreducible components of $X$, 
to the case where $X$ is integral.
Take an open immersion from $X$ to a 
connected scheme $X'$ of dimension $d$,
which is proper over $\Spec \F_q$, such that $X$ is dense in $X'$.
Take a non-empty affine open subscheme $U \subset X$ and set
$Y=X' \setminus U$. 
Let us take an algebraic closure $\Fbar_q$ of $\F_q$.
By \cite[Chapter II, Proposition 3.1, p.~66]{Hartshorne}
and \cite[Chapter II, Proposition, p.~67]{Hartshorne} 
(originally due to \cite{Goodman}),
for each irreducible component $X''$ of 
$X' \times_{\Spec \, \F_q} \Spec \, \Fbar_q$, 
we know that $X''\setminus U \times_{\Spec \, \F_q} \Spec \, \Fbar_q$ 
is connected and is of pure codimension one in $X'$.
This shows that $Y$ is of 
pure codimension one in $X'$.

Let us show that $Y$ is connected.
Write $f \colon X '\times_{\Spec \F_q} \Spec \Fbar_q
\to X'$ for the canonical projection.
We note that the map $f$ is surjective, 
and, as the canonical morphism $\Spec \Fbar_q
\to \Spec \F_q$ is universally closed 
by \cite[Proposition (6.1.10)]{EGA2},
the map $f$ is a closed map.
Let $\xi \in X$ denote the generic point of $X$.
As $X$ is dense in $X'$, 
the closure of $\xi$ in $X'$ equals $X'$.
Take $\xi' \in f^{-1}(\xi)$ and let 
$X''$ be an irreducible component of 
$X' \times_{\Spec \F_q} \Spec \Fbar_q$
that contains $\xi'$.  
Using that an irreducible component is closed, 
we see that $X''$ contains the closure in 
$X' \times_{\Spec \F_q}\Spec \Fbar_q$ 
of $\xi'$.
Then as $f$ is a closed map,
the morphism $f|_{X''}:X'' \to X'$ is surjective.
Using the fact above by Goodman, 
we have that 
$X'' \setminus U\times_{\Spec \F_q} \Spec \Fbar_q$
is connected.
Then as $X'' \setminus (U \times_{\Spec \F_q} \Spec \Fbar_q)$
surjects onto $X' \setminus U=Y$,
we have that $Y$ is connected as the continuous image of a connected 
space.

Write $X\cap Y=Z_1 \coprod \dots \coprod Z_r$ so that each $Z_i$ is connected.
We claim that each $Z_i$ is not proper.
As $X \subset X'$ is an 
open subset, $X\cap Y \subset Y$ is an open subset of $Y$,
hence each $Z_i \subset Y$ is an open subset of $Y$.
As $Y$ is connected, $\overline{Z}^Y_i \neq Z_i$
where $\overline{Z}^Y_i$ denotes the closure of 
$Z_i$ in $Y$.   This implies that
$Z_i$ is the complement of a non-empty closed set,
namely $\overline{Z}^Y_i \setminus Z_i$,
of a connected proper scheme $\overline{Z}^Y_i$.
It follows from Lemma~\ref{lem:not proper} that 
$Z_i$ is not proper.

Let $j \in \{-1, -2\}$.
Since $U$ is affine, from the localization sequence 
$$
H^\BM_{-1}(Y \cap X, \Z(j))
\to H^\BM_{-1}(X,\Z(j))
\to H^\BM_{-1}(U,\Z(j)) 
$$
it follows by induction on $d$ that
$H^\BM_{-1}(X,\Z(j))$ is zero
(to remove the hypothesis that the schemes in the
localization sequence are quasi-projective, we refer to
\cite[Theorem 1.7, p.~301]{Levine2} and \cite[2.6, p.~60]{Ge-Le2}).
This proves the claim for $X$ not proper.

The claim for $X$ proper follows from
Lemmas~\ref{lem:red_to_non_proper} and \ref{lem:red_to_normal}.
This completes the proof.
\end{proof}

\section{Under Parshin's conjecture}
\label{sec:Parshin}
We assume Parshin's conjecture in this section
and draw some consequences.
Parshin's conjecture states that
for any projective smooth scheme $Z$ over a finite field,
$H_\cM^a(Z, \Q(b))=0$ unless $a=2b$.
We note that it is a theorem of Harder for curves 
(we refer to \cite[THEOREM~0.5, p.70]{Grayson} for the 
correct implication of Harder's result).
\begin{prop}
\label{prop:Parshin}
Assume that Parshin's conjecture holds.
Then
the statement in 
Theorem~\ref{prop:AppB_main} holds true for any $j \le -1$.
We also have  $H_i^\BM(X, \Z(j))=0$ for $i \le -2$ and $j \le -1$.
\end{prop}

We begin with a lemma.
\begin{lem}
\label{lem:vanishing}
\begin{enumerate}
\item 
Let $U$ be a connected scheme of pure dimension $d\ge 1$
over $\F_q$.  
Then
$\varinjlim_V H_i^\BM(V, \Z(j))=0$, 
where $V$ runs over the (non-empty) open subschemes 
of $U$,
for $i \le -1$ and $j\le -1$ assuming Parshin's conjecture.
\item
Let $V$ be a zero dimensional scheme over $\F_q$.
Then
$H_i^\BM(V, \Z(j))=0$
for $i \le -2$ and $j \le -1$ assuming Parshin's conjecture.
\end{enumerate}
\end{lem}
\begin{proof}
Let $K$ denote the function field of $U$.
If $d>i-j$, then
$H_i^\BM(V, \Z(j))=\CH_j(V, i-2j)=\CH^{d-j}(V, i-2j)$.
So the limit is $\CH^{d-j}(\Spec K, i-2j)$, which equals zero 
by dimension reason.

Suppose $d\le i-j$.
Let $V$ be an open smooth subscheme of $U$.
We proceed as in the proof of Lemma~\ref{lem:q ge 2}.
We have
\[
H_i^\BM(V, \Z(j))=H_\cM^{2d-i}(V, \Z(d-j))
=H_\cM^{2d-i-1}(V, \Q/\Z(d-j)),
\]
where the second equality follows 
from \cite[Theorem 4.7 ii), p.312]{Geisser2}
(this uses Parshin's conjecture).
Since we are in the range $d \le i-j$,
we can use \cite[Corollary~1.2, p.56]{Ge-Le2} 
and \cite[Theorem~8.4, p.491]{Ge-Le1} 
to see that the quantity above is isomorphic to
$\bigoplus_{\ell\neq p}H_{\mathrm{et}}^{2d-i-1}(V, \Q_\ell/\Z_\ell(d-j))
\oplus \varinjlim_{r}H^{d+j-i-1}(X_\mathrm{Zar}, \nu_r^{d-j})$.
The $p$-part is zero since we are in the range $d \le i-j$.

We may assume that $V$ is affine.
Then $H_{\mathrm{et}}^{2d-i-1}(V, \Q_\ell/\Z_\ell(d-j))=0$
for $d-3 \ge i$ since the cohomological dimension 
of $V$ is $d+1$ 
(\cite[Expos\'e XIV, Th\'eor\`eme~3.1, p.15]{SGA4}).
Suppose $d\ge 2$.  Then the claim follows from 
this immediately since $i \le -1$.
Suppose $d=1$.
The vanishing  
$H^2_\mathrm{et}(V, \Q_\ell/\Z_\ell(1-j))=0$
can be shown using the same method as 
in the proof of Lemma~\ref{lem:q ge 2}.
This proves (1).
Let $V$ be as in (2).
The remaining case is $i=-2$.
We proceed as in the proof of Lemma~\ref{lem:q ge 2}.
We have an exact sequence
\[
H_\mathrm{et}^1(V, \Q_\ell(-j)) 
\to H_\mathrm{et}^1(V, \Q_\ell/\Z_\ell(-j))
\to
H_\mathrm{et}^2(V, \Z_\ell(-j)).
\]
The third term is zero since $V$ is zero dimensional.
We use the Hochschild-Serre spectral sequence
as before.  We have $E_2^{0,1}=E_2^{1,0}=0$
using the weight argument.
The claim then follows.
This completes the proof.
\end{proof}

\begin{proof}[Proof of Proposition~\ref{prop:Parshin}]
Suppose $i=-1$.
For the proof of Theorem~\ref{prop:AppB_main} to work 
for general $i$, we need as an input 
the vanishing of 
$\varinjlim_Y H_{-1}^\BM(X \setminus Y, \Z(j))$
(the notation as in the proof of Theorem~\ref{prop:AppB_main}), 
and this is all that we need.
As we have seen in Lemma~\ref{lem:vanishing}
above, the vanishing holds under Parshin's conjecture,
hence the claim follows if $i=-1$.

Suppose $i \le -2$.  We show by induction on the dimension $d$
that $H_i^\BM(X, \Z(j))=0$.
The case $d=0$ is Lemma~\ref{lem:vanishing}(2).
Consider the exact sequence
\[
\varinjlim_Y H_i^\BM(Y, \Z(j))
\to
H_i^\BM(X, \Z(j))
\to
\varinjlim_Y H_i^\BM(X\setminus Y, \Z(j))
\]
where $Y$ runs over closed subschemes of $X$.
By the inductive hypothesis, 
the first term is zero,
and the third term is zero by Lemma~\ref{lem:vanishing}(1).
\end{proof}

For fixed $(i,j)$, we only need to
assume Parshin's conjecture for projective smooth schemes of dimension 
(less than or equal to) $i-j$.
Theorem~\ref{prop:AppB_main} treats the cases $(i,j)=(-1,-1)$ and $(-1, -2)$.
Hence we can use Harder's result and need not assume Parshin's conjecture.

\begin{rmk}\normalfont
\label{rmk:BMetale}
For this remark, 
we do not use
Parshin's conjecture.
Let us define and compute the 
\'etale Borel-Moore (not motivic) homology groups 
for a scheme $X$ separated and of finite type over $\F_q$
in the same range as that of Proposition~\ref{prop:Parshin}.
We will see that the \'etale Borel-Moore homology groups
and the Borel-Moore motivic homology groups
are isomorphic in this range.
Let $\ell$ be a prime number prime to $p$.
We define the \'etale Borel-Moore homology group
to be
\[
H_i^{\BM,\mathrm{et}}(X, \Z/\ell^n(j)) =
H_{\mathrm{et}}^{-i}(X, Rf^!(\Z/\ell^n(-j)))
\]
for $i,j\in \Z$ and $n \ge 1$.
Since this is isomorphic to
\[
\Hom_{\Z/\ell^n}
(H_{\mathrm{et}, c}^{i+1}
(X, \Z/\ell^n(-j)), \Z/\ell^n),
\]
we set
\[
H_i^{\BM,\mathrm{et}}(X, \Z_\ell(j))
=\Hom_{\Q_\ell/\Z_\ell}
(H_{\mathrm{et},c}^{i+1}(X, \Q_\ell/\Z_\ell(-j)), \Q_\ell/\Z_\ell).
\]
Then it is easy to see that a statement similar to the one in Proposition~\ref{prop:Parshin}
holds for \'etale Borel-Moore (not motivic) 
homology groups with $\Z_\ell$-coefficient.
Namely,
we have $H_i^{\BM,\mathrm{et}}(X, \Z_\ell(j))=0$ for $i \le -2$,
and, for $X$ connected and for $i=-1$, 
the pushforward  by the structure morphism is an isomorphism if $X$
is proper, and $H_{-1}^{\BM,\mathrm{et}}(X, \Z_\ell(j))=0$ otherwise.
\end{rmk}

\end{document}